\documentclass[11pt]{amsart}
\usepackage{amssymb,amsthm}
\usepackage{mathtools}
\usepackage{eucal,stmaryrd}
\usepackage{enumerate}
\usepackage[margin=1in]{geometry}
\usepackage{quiver}

\theoremstyle{plain}
\newtheorem{theorem}{Theorem}[section]
\newtheorem{proposition}[theorem]{Proposition}
\newtheorem{lemma}[theorem]{Lemma}
\newtheorem{corollary}[theorem]{Corollary}

\theoremstyle{definition}
\newtheorem{definition}[theorem]{Definition}
\newtheorem{example}[theorem]{Example}

\usepackage{hyperref}
\usepackage{xcolor}
\hypersetup{
    colorlinks,
    linkcolor={red!50!black},
    citecolor={blue!50!black},
    urlcolor={blue!80!black}
}

\newcommand{\tp}{{\scriptscriptstyle\mathsf{T}}}
\newcommand{\pp}{{\scriptscriptstyle++}}

\let\oh=\circ
\newcommand{\ccirc}{\mathbin{\mathchoice
  {\xcirc\scriptstyle}
  {\xcirc\scriptstyle}
  {\xcirc\scriptscriptstyle}
  {\xcirc\scriptscriptstyle}
}}
\newcommand{\xcirc}[1]{\vcenter{\hbox{$#1\oh$}}}
\let\circ\ccirc

\newcommand{\lb}{\llbracket}
\newcommand{\rb}{\rrbracket}

\let\O\undefined
\let\S\undefined
\let\P\undefined

\DeclareMathOperator{\O}{O}
\DeclareMathOperator{\V}{V}

\DeclareMathOperator{\S}{S}
\DeclareMathOperator{\SO}{SO}
\DeclareMathOperator{\Alt}{\mathsf{\Lambda}}
\DeclareMathOperator{\Sym}{\mathsf{S}}
\DeclareMathOperator{\tr}{tr}
\DeclareMathOperator{\Gr}{Gr}
\DeclareMathOperator{\diag}{diag}
\DeclareMathOperator{\spn}{span}
\DeclareMathOperator{\Flag}{Flag}
\DeclareMathOperator{\Fl}{Fl}
\DeclareMathOperator{\GL}{GL}
\DeclareMathOperator{\P}{P}
\DeclareMathOperator{\sign}{sign}
\DeclareMathOperator{\Ad}{Ad}

\begin{document}
\title{Simple matrix models for the flag, Grassmann, and Stiefel manifolds}
\author[L.-H.~Lim]{Lek-Heng~Lim}
\address{Computational and Applied Mathematics Initiative, Department of Statistics,
University of Chicago, Chicago, IL 60637-1514.}
\email{lekheng@uchicago.edu}
\author[K.~Ye]{Ke Ye}
\address{KLMM, Academy of Mathematics and Systems Science, Chinese Academy of Sciences, Beijing 100190, China}
\email{keyk@amss.ac.cn}

\begin{abstract}
We derive three families of orthogonally-equivariant matrix submanifold models for the Grassmann, flag, and Stiefel manifolds respectively.  These families are exhaustive --- every orthogonally-equivariant submanifold model of the lowest dimension for any of these manifolds is necessarily a member of the respective family, with a small number of exceptions. They have several computationally desirable features. The orthogonal equivariance allows one to obtain, for various differential geometric objects and operations, closed-form analytic expressions that are readily computable with standard numerical linear algebra. The minimal dimension aspect translates directly to a speed advantage in computations. And having an exhaustive list of all possible matrix models permits one to identify the model with the lowest matrix condition number, which translates to an accuracy advantage in computations. As an interesting aside, we will see that the family of models for the Stiefel manifold is naturally parameterized by the Cartan manifold, i.e., the positive definite cone equipped with its natural Riemannian metric.
\end{abstract}

\maketitle

\section{Introduction}\label{sec:intro}

As abstract manifolds, the Grassmannian is a set of subspaces, the flag manifold a set of nested sequences of subspaces, and the Stiefel manifold a set of  frames. To work with these manifolds, not least performing computations, one needs a model, i.e., a system of extrinsic coordinates, for them. From the perspective of numerical computations, the best models are matrix models, representing these manifolds by a quotient manifold or submanifold of matrices, thereby permitting differential geometric objects to be expressed in terms of matrix operations computable with standard numerical linear algebra. Among such models, submanifold models are preferable to quotient manifold models as points and tangent vectors are represented by actual matrices instead of equivalence classes of matrices, which require artificial choices and additional computations to further represent as actual matrices, bearing in mind that standard numerical linear algebra only works with actual matrices, not equivalence classes of matrices. For numerical stability, these models should be equivariant under orthogonal group action, as orthogonal matrices are the basis of numerically stable algorithms.

For the Grassmannian and flag manifold, these considerations lead us to the following models. Let $\Sym^2(\mathbb{R}^n)$ be the space of $n \times n$ real symmetric matrices, $a,b \in \mathbb{R}$ distinct and $a_1,\dots,  a_{p+1} \in \mathbb{R}$  generic real numbers. We will show that the \emph{quadratic model}
\[
 \Gr_{a,b}(k,n)\coloneqq
\bigl\lbrace
X \in \Sym^2(\mathbb{R}^n)  : (X - aI) (X - bI) = 0,  \; \tr(X) = ak + b(n-k)
\bigr\rbrace
\]
is diffeomorphic to $\Gr(k, \mathbb{R}^n)$, the Grassmannian of $k$-planes in $\mathbb{R}^n$; and the \emph{isospectral model}
\[
\Flag_{a_1,\dots,  a_{p+1}}(k_1,\dots,  k_p, n)  \coloneqq
\biggl\lbrace
X \in \Sym^2(\mathbb{R}^n) : \prod_{j=0}^p (X - a_{j+1} I_n) = 0,  \; \tr(X) = \sum_{j=0}^p (k_{j+1} - k_j) a_{j+1}
\biggr\rbrace
\]
is diffeomorphic to $\Flag(k_1,\dots,k_p, \mathbb{R}^n)$, the manifold of $(k_1,\dots,k_p)$-flags in $\mathbb{R}^n$. Here we have assumed that $0 < k <n$ and $0\eqqcolon k_0 < k_1 < \dots < k_{p+1} \coloneqq n$ are all integers.  Evidently the quadratic model is the $p = 1$ case of the isospectral model.

The quadratic model is so-called because the roots of a monic quadratic matrix polynomial $(X - aI) (X - bI) = 0$ are called quadratic matrices \cite{quad1}, well-known in studies of numerical range \cite{quad3,quad2}. The special cases $(a,b)=(1,0)$ gives projection matrices with $X^2 = X$ and $(a,b)=(1,-1)$ gives involutory matrices with $X^2 = I$. So the classical projection model $\Gr_{1,0}(k,n) = \{X \in \Sym^2(\mathbb{R}^n) : X^2 =X,  \; \tr(X) = k \}$ \cite{Mattila,Nicolaescu} and the more recent involution model $\Gr_{1,-1}(k,n) =  \{X \in \Sym^2(\mathbb{R}^n) : X^2 =I,  \; \tr(X) = 2k - n\}$ \cite{ZLK20} are both special cases of the quadratic model.

The descriptions of these models above are implicit. We will prove in Section~\ref{sec:or} that they are equivalent to the following explicit descriptions:
\begin{align*}
 \Gr_{a,b}(k,n) &=
\biggl\lbrace
Q \begin{bmatrix} a I_k & 0 \\ 0 & b I_{n-k} \end{bmatrix} Q^\tp \in \Sym^2(\mathbb{R}^n) :  Q \in \SO_n(\mathbb{R}) \biggr\rbrace,
\\
\Flag_{a_1,\dots,  a_{p+1}}(k_1,\dots,  k_p, n) &=\left\lbrace
Q {\setlength{\arraycolsep}{0pt}
\begin{bmatrix}
a_1 I_{n_1} & 0 & \cdots & 0\\[-1ex]
0 &  a_2 I_{n_2} &  &  \vdots \\[-1ex]
\vdots &  & \ddots &  0 \\
0 & \cdots & 0 &  a_{p+1} I_{n_{p+1}}
\end{bmatrix}}
Q^\tp\in \Sym^2(\mathbb{R}^n) : Q\in \SO_n(\mathbb{R}) 
\right\rbrace,
\end{align*}
where $\SO_n(\mathbb{R})$ is the set of orthogonal matrices with unit determinant. This matrix manifold has been in existence for more than thirty years \cite{Brockett91,  CD91, LM08, AM12,  FL19} and is known by the names \emph{isospecral manifold} or \emph{spectral manifold}. What is not known is that it is just a parameterization of the flag manifold, which is of course classical and known for nearly 120 years since \cite{dS} or, in its modern form,  for at least fifty years since \cite[Section~5]{borel}.  To  the best of our knowledge, the connection that they are one and the same has never been made, a minor side contribution of this article.

One might also ask how these discussions apply to the Stiefel manifold \cite{Stie} of orthonormal $k$-frames in $n$-space. As an addendum, we will show in Section~\ref{sec:V} that for the Stiefel manifold, there is an analogous family of minimal $\SO_n(\mathbb{R})$-equivariant models, the \emph{Cholesky models}
\[
\V_{\!A}(k,n) \coloneqq \{ X \in \mathbb{R}^{n \times k} : X^\tp X = A \},
\]
parameterized by $A \in \Sym^2_\pp(\mathbb{R}^k)$, the set of $k \times k$ symmetric positive definite matrices. This description is implicit; but Cholesky models too have an explicit description that justifies the name,
\[
 \V_{\!A} (k,n) = 
\biggl\lbrace
Q \begin{bmatrix}
R \\
0
\end{bmatrix}: Q\in \O_n(\mathbb{R})
\biggr\rbrace,
\]
where $R \in \GL_k(\mathbb{R})$ is the unique Cholesky factor with positive diagonal of $A$.

It will become clear in Section~\ref{sec:V} that the relevant structure on $\Sym^2_\pp(\mathbb{R}^k)$ is that of a Riemannian manifold with metric $\tr(A^{-1}XA^{-1}Y)$, and not its more common structure as an Euclidean cone. Note that $A = I$ gives the usual model of the Stiefel manifold as $n \times k$ matrices with orthonormal columns but more generally a Cholesky model allows for $A^{-1}$-orthonormal columns.\footnote{Fortuitously Stiefel also pioneered the use of $A^{-1}$-orthonormality in computational mathematics through his conjugate gradient method \cite{CG}.}

\subsection{Computational significance}

From a computational perspective, these families of models have two desirable features and are unique in this regard:
\begin{enumerate}[\upshape (i)]
\item\label{it:or} \textsc{orthogonal equivariance}: Let $Q \in \O(n)$. Then $X \in  \Sym^2(\mathbb{R}^n)$ is in an isospectral model iff $Q^\tp X Q$ is in the model; $X \in \mathbb{R}^{n \times k}$ is in a Cholesky model iff $QX$ is in the model.
\item\label{it:em} \textsc{minimal dimension}: There is no model for the flag (resp.\ Stiefel) manifold in an ambient space of dimension smaller than $\frac12 (n-1)(n+2)$ (resp.\ $nk$) with property~\eqref{it:or}.
\item\label{it:un} \textsc{exhaustive}: Any model for the flag (resp.\ Stiefel) manifold with properties \eqref{it:or} and \eqref{it:em} must be among the  family of isospectral (resp.\ Cholesky) models.
\end{enumerate}

Property~\eqref{it:or} is key to deriving closed-form analytic expressions for differential geometric quantities in terms of standard matrix operations  stably computable with numerical linear algebra \cite{ZLK20}.

Property~\eqref{it:em} ensures that points on the manifold are represented with matrices of the lowest possible dimension, which is important as the \emph{speed} of every algorithm in numerical linear algebra depends on the dimension of the matrices. The current lowest dimensional matrix model of a flag manifold is obtained by embedding $\Flag(k_1,\dots,  k_p;\mathbb{R}^n)$ into a product of Grassmannians as in \eqref{eq:flag3} on p.~\pageref{eq:flag3}. Even if we use the lowest dimensional models for these Grassmannians, the result would involve matrices of dimension $np \times np$ whereas the isospectral model uses only $n \times n$ matrices.

For \emph{accuracy}, the matrix condition number plays a role analogous to matrix dimension for speed. Property~\eqref{it:un} allows us to pick from among the respective family of models the one with the smallest condition number. Every matrix in an isospectral model $\Flag_{a_1,\dots,  a_{p+1}}(k_1,\dots,  k_p, n)$ has identical eigenvalues and therefore condition number $\max (\lvert a_1 \rvert,\dots,\lvert a_{p+1} \rvert) / \min (\lvert a_1 \rvert,\dots,\lvert a_{p+1} \rvert)$. For the Grassmannian $p = 1$ and there is a unique (up to a constant multiple) perfectly conditioned model with $a=1$, $b=-1$, which is precisely the involution model in \cite{ZLK20}. For more general flag manifolds with $p > 1$, the condition number can be made arbitrarily close to one. For the Stiefel manifold, the usual model with $A = I$ plays the role of the involution, namely, as the unique (up to a constant multiple) perfectly conditioned model among all Cholesky models.

\subsection{Contributions}

The main intellectual effort of this article is to establish property~\eqref{it:or} and half of property~\eqref{it:em} by deriving the isospectral model (Theorem~\ref{thm:classification}, Corollary~\ref{cor:isospectral}) and demonstrating that we may choose $a_1,\dots,  a_{p+1}$ so that we get a submanifold of  $\Sym^2_\oh(\mathbb{R}^n)$, the space of traceless symmetric matrices with $\dim \Sym^2_\oh(\mathbb{R}^n) = \frac12 (n-1)(n+2)$  (Corollary~\ref{cor:quad}). To demonstrate the other half of property~\eqref{it:em}, i.e., no lower dimensional model exists, and to prove property~\eqref{it:un}, one needs an amount of representation theory far beyond the scope of our article and is relegated to \cite{LK24b}.

The orthogonal equivariance in property~\eqref{it:or} deserves elaboration. Firstly we really do mean ``orthogonal'' --- every result in this article remains true if $\SO_n(\mathbb{R})$ is replaced by $\O_n(\mathbb{R})$. Secondly we stress the importance of ``equivariance.'' The Riemannian metric is often regarded as the center piece of any computations involving manifolds, not least in manifold optimization. This is getting things backwards.  What is by far more important is equivariance, or, as a special case, \emph{invariance}. There are uncountably many Riemannian metrics on the flag manifold even after discounting constant multiples. The most important Riemannian metrics are the ones that are $\SO_n(\mathbb{R})$-invariant; note that the standard Euclidean inner product on $\mathbb{R}^n$ or $\mathbb{R}^{m \times n}$ has this property. In the case of the flag, Grassmann, and Stiefel manifolds, there is an even more special one --- the $\SO_n(\mathbb{R})$-invariant Riemannian metric that comes from the unique bi-invariant metric on $\SO_n(\mathbb{R})$. It is the key to deriving explicit expressions for differential geometric objects in terms of standard matrix operations readily computable with numerical linear algebra.

A second goal of this article is show that the Riemannian metrics arising from our equivariant embeddings of the flag, Grassmann, and Stiefel manifolds are, up to constant weights, the ones arising from the bi-invariant metric on $\SO_n(\mathbb{R})$ (Section~\ref{sec:metric}).  As a perusal of the computational mathematics literature would reveal, equivariance has never been brought into center stage. We hope that our article would represent a small step in this direction.

\section{Notations and some background}

We generally use blackboard bold fonts for vector spaces and double brackets $\lb \, \cdot \, \rb$ for equivalence classes. On two occasions we write $\mathbb{Q}_j$ for a subspace spanned by columns of an orthogonal matrix $Q$, which should not cause confusion as the rational field has no role in this article. We reserve the letter $\mathbb{W}$ for $\mathbb{R}$-vector spaces and $\mathbb{V}$ for $\SO_n(\mathbb{R})$-modules, usually adorned with various subscripts. We write $\cong$ for diffeomorphisms of manifolds and $\simeq$ for isomorphisms of vector spaces and modules. 

\subsection{Linear algebra}\label{sec:mat}

The real vector spaces of real symmetric, skew-symmetric, and traceless symmetric matrices will be denoted respectively by
\begin{equation}\label{eq:space}
\begin{aligned}
\Sym^2(\mathbb{R}^n) &\coloneqq \{ X \in \mathbb{R}^{n \times n} : X^\tp = X \},\\
\Alt^2(\mathbb{R}^n) &\coloneqq \{ X \in \mathbb{R}^{n \times n} : X^\tp = -X \},\\
\Sym^2_\oh(\mathbb{R}^n) &\coloneqq \{ X \in \mathbb{R}^{n \times n} : X^\tp = X,\; \tr(X) = 0 \}.
\end{aligned}
\end{equation}
For the cone of real symmetric positive definite matrices, we write
\[
\Sym^2_\pp(\mathbb{R}^n) \coloneqq \{ X \in \mathbb{R}^{n \times n} : y^\tp X y > 0 \text{ for all } y \ne 0\}.
\]
The Lie groups of real orthogonal, special orthogonal, and general linear groups will be denoted  respectively by
\begin{equation}\label{eq:group}
\begin{aligned}
\O_n(\mathbb{R}) &\coloneqq \{ X \in \mathbb{R}^{n \times n} : X^\tp X = I \},\\
\SO_n(\mathbb{R}) &\coloneqq \{ X \in \mathbb{R}^{n \times n} : X^\tp X = I, \; \det(X) = 1  \},\\
\GL_n(\mathbb{R}) &\coloneqq \{ X \in \mathbb{R}^{n \times n} : \det(X) \ne 0\}.
\end{aligned}
\end{equation}
Let $n_1 + \dots + n_p = n$. We will also write
\begin{multline*}
\S(\O_{n_1}(\mathbb{R}) \times \cdots \times \O_{n_p}(\mathbb{R}) ) \coloneqq \{ \diag(X_1,\dots,X_p) \in \O_n(\mathbb{R}) :\\
 X_1 \in \O_{n_1}(\mathbb{R}),\dots, X_p \in \O_{n_p}(\mathbb{R}), \; \det( X_1) \cdots \det(X_p) = 1 \}.
\end{multline*}
Note that $\SO_n(\mathbb{R})$ is a special case of this. For each increasing sequence $0\eqqcolon k_0 < k_1 < \cdots < k_p < k_{p+1} \coloneqq n$,  we also define the \emph{parabolic subgroup} of block upper triangular matrices:
\[
\P_{k_1,\dots,  k_p}(\mathbb{R}) = \Biggl\lbrace
\begin{bsmallmatrix}
X_{1,1} & \cdots & X_{1,p+1} \\
 \vdots & \ddots & \vdots \\
 0& \cdots & X_{p+1,p+1} \\
\end{bsmallmatrix}\in \GL_n(\mathbb{R}): \begin{multlined} X_{ij} \in \mathbb{R}^{(k_{i} - k_{i-1}) \times (k_j - k_{j-1})},\\ i,j = 1,\dots, p+1 \end{multlined}
\Biggr\rbrace.
\]

Let $G$ be a group and $\mathbb{V}$ a vector space.  We say that $\mathbb{V}$ is a \emph{$G$-module} if there is a \emph{linear} group action $G \times \mathbb{V} \to \mathbb{V}$, $(g,v) \mapsto g \cdot v$, i.e., satisfying
\[
g \cdot (a_1 v_1 + a_2 v_2) = a_1 g \cdot v_1 + a_2 g \cdot v_2
\]
for any $g\in G$,  $a_1,a_2\in \mathbb{R}$, and $v_1,v_2\in \mathbb{V}$.  In this paper, we will mostly limit ourselves to $G = \SO_n(\mathbb{R})$ and two group actions on $\mathbb{R}^n$ and $\mathbb{R}^{n\times n}$ respectively:
\begin{alignat}{2}
\SO_n(\mathbb{R}) \times  \mathbb{R}^n &\to  \mathbb{R}^n, \quad &(Q,  v) &\mapsto Q \cdot v \coloneqq Q v,\label{eq:leftmult} \\
\SO_n(\mathbb{R}) \times  \mathbb{R}^{n \times n} &\to  \mathbb{R}^{n \times n}, \quad &(Q,  X) &\mapsto Q \cdot X \coloneqq Q X Q^\tp.  \label{eq:conj}
\end{alignat}
All matrix subspaces in \eqref{eq:space} are also $\SO_n(\mathbb{R})$-modules under the conjugation action \eqref{eq:conj}. If $\mathbb{V}$ is a $G$-module, then a direct sum of multiple copies $\mathbb{V} \oplus \dots \oplus \mathbb{V}$ is automatically a $G$-module with action $Q \cdot (v_1,\dots,v_k) \coloneqq 
(Q \cdot v_1,\dots, Q \cdot v_k)$. So $\mathbb{R}^{n \times k} =  \mathbb{R}^n \oplus \dots \oplus  \mathbb{R}^n$ is also an $\SO_n(\mathbb{R})$-module under the action \eqref{eq:leftmult}.

\subsection{Differential geometry} We write $\Gr(k,\mathbb{R}^n)$, $\Flag(k_1,\dots,k_p,\mathbb{R}^n)$, and $\V(k,\mathbb{R}^n)$ respectively for the  flag, Grassmann, and Stiefel manifold as \emph{abstract manifolds}. An element of $\Gr(k,\mathbb{R}^n)$ is a subspace  $\mathbb{W} \subseteq \mathbb{R}^n$, $\dim \mathbb{W} = k$. An element of $\Flag(k_1,\dots,k_p,\mathbb{R}^n)$  is a flag, i.e., a nested sequence of subspaces $\mathbb{W}_1 \subseteq \dots \subseteq \mathbb{W}_p \subseteq \mathbb{R}^n$, $\dim \mathbb{W}_i = k_i$, $i=1,\dots,p$. An element of $\V(k,\mathbb{R}^n)$ is an orthonormal $k$-frame $(w_1,\dots,w_k)$ in $\mathbb{R}^n$.

The abstract Grassmann and flag manifolds are \emph{$G$-manifolds}  for any subgroup $G \subseteq \GL_n(\mathbb{R})$, i.e., they come naturally equipped with a $G$-action. For Grassmann manifolds, take any $\mathbb{W} \in  \Gr(k, \mathbb{R}^n)$ and any $X \in G \subseteq \GL_n(\mathbb{R})$, the action is given by
\[
X \cdot \mathbb{W}  \coloneqq \{ Xw \in \mathbb{R}^n : w \in \mathbb{W}\},
\]
noting that $\dim X \cdot \mathbb{W}  = \dim  \mathbb{W}$. This action extends to $\Flag(k_1,\dots,  k_p,\mathbb{R}^n)$ since if  $\mathbb{W}_1 \subseteq \dots \subseteq \mathbb{W}_p $, then  $X \cdot \mathbb{W}_1 \subseteq \dots \subseteq X \cdot \mathbb{W}_p $ for any $X \in G \subseteq \GL_n(\mathbb{R})$. So we may write
\begin{equation}\label{eq:ac}
X \cdot (\mathbb{W}_1 \subseteq \dots \subseteq \mathbb{W}_p ) \coloneqq (X \cdot \mathbb{W}_1 \subseteq \dots \subseteq  X \cdot \mathbb{W}_p),
\end{equation}
again noting that since dimensions are preserved we have a well-defined action. If $G$ is any of the groups in \eqref{eq:group}, then this action is transitive.

The abstract Stiefel manifold is a $G$-manifold for any subgroup $G \subseteq \O_n(\mathbb{R})$ via the action
\begin{equation}\label{eq:acV}
X \cdot (w_1,\dots,w_k) \coloneqq (Xw_1,\dots, Xw_k)
\end{equation}
for any $(w_1,\dots,w_k) \in \V(k,\mathbb{R}^n)$. This action is transitive if $G = \O_n(\mathbb{R})$ or $\SO_n(\mathbb{R})$.

A notion central to this article is that of an equivariant embedding, which has been extensively studied in a more general setting \cite{Mostow57,  Bredon72,  LV83, Timashev11}.  
\begin{definition}[Equivariant embedding and equivariant submanifold]
Let $G$ be a group, $\mathbb{V}$ be a $G$-module, and $\mathcal{M}$ be a $G$-manifold. An embedding $\varepsilon: \mathcal{M} \to \mathbb{V}$ is called a $G$-equivariant embedding if $\varepsilon(g \cdot x) = g \cdot \varepsilon(x)$ for all $x\in \mathcal{M}$ and $g \in G$.  In this case, the embedded image $\varepsilon(\mathcal{M})$ is called a $G$-submanifold of $\mathbb{V}$.
\end{definition}
For this article, we only need to know two special cases. For the flag manifold (and thus Grassmannian when $p =1$), $G = \SO_n(\mathbb{R})$, $\mathcal{M} = \Flag(k_1,\dots,k_p, \mathbb{R}^n)$, $\mathbb{V} = \mathbb{R}^{n \times n}$, $G$ acts on  $\mathbb{V}$ via \eqref{eq:conj}  and on $\mathcal{M}$ via \eqref{eq:ac}.  For the Stiefel manifold,  $G = \SO_n(\mathbb{R})$,  $\mathcal{M} = \V(k,n)$,  $\mathbb{V} = \mathbb{R}^{n\times k}$,  $G$ acts on $\mathbb{V}$ via \eqref{eq:leftmult} and on $\mathcal{M}$ via \eqref{eq:acV}.

\section{Existing models of the flag manifold}

The reader may find a list of all known existing matrix models of the Grassmannian and Stiefel manifold  in \cite{ZLK24}.
Models for the flag manifold are more obscure and we review a few relevant ones here. The model most commonly used in pure mathematics is as a quotient manifold,
\begin{equation}\label{eq:flag1}
\Flag(k_1,\dots,  k_p,\mathbb{R}^n) \cong \GL_{n}(\mathbb{R})/\P_{k_1,\dots,  k_p}(\mathbb{R}),
\end{equation}
with the parabolic subgroup as defined in Section~\ref{sec:mat}.

As for known\footnote{The authors of \cite{Brockett91,  CD91, LM08, AM12,  FL19} did not appear to know that they are discussing the flag manifold.} models of the flag manifold in applied mathematics, the only ones we are aware of were first proposed in \cite{KKL22}. We will highlight two: There is the orthogonal analogue of \eqref{eq:flag1}, as the quotient manifold
\begin{equation}\label{eq:flag2}
\Flag(k_1,\dots,  k_p,\mathbb{R}^n) \cong \SO_{n}(\mathbb{R})/\S(\O_{n_1}(\mathbb{R}) \times \cdots \times \O_{n_{p+1}}(\mathbb{R}) ).
\end{equation}
It has been shown in \cite{KKL22} that any  flag manifold may be embedded as a submanifold in a product of Grassmannians so that any model for the latter yields a model for the former. We recall this result here: Let $n_1,  \dots,  n_{p+1} \in \mathbb{N}$ with $n_1 + \dots + n_{p+1} = n$. Define
\begin{multline}\label{eq:flag3}
\Flag_{\Gr}(n_1,\dots,n_{p+1}) \coloneqq
\bigl\lbrace (\mathbb{W}_1,\dots,  \mathbb{W}_{p+1}) \in \Gr(n_1, \mathbb{R}^n) \times \dots \times \Gr(n_{p+1}, \mathbb{R}^n) :\\
 \mathbb{W}_1 \oplus \dots \oplus \mathbb{W}_{p+1} = \mathbb{R}^n \bigr\rbrace,
\end{multline}
with $\oplus$ the orthogonal direct sum of subspaces. Then every flag manifold is diffeomorphic to a  submanifold of the form in \eqref{eq:flag3} via the following lemma, where we have also included \eqref{eq:flag1} and \eqref{eq:flag2} for completeness. Although we use $\varphi_1$ for a different purpose in this article, its existence carries the important implication that every model of the Grassmannian automatically gives one for the flag manifold.
\begin{lemma}[Change-of-coordinates for flag manifolds I]\label{lem:diffeo}
Let $0 \eqqcolon k_0  <   k_1 < \cdots <  k_p  < k_{p+1} \coloneqq n$ be integers and
\begin{equation}\label{eq:n}
n_{i+1} = k_{i+1} - k_i,\quad i =1,\dots, p.
\end{equation}
Then the maps below are all diffeomorphisms:
\[\begin{tikzcd}
	{\Flag(k_1,\dots,  k_p,\mathbb{R}^n)} && {\Flag_{\Gr}(n_1,\dots,n_{p+1})} \\
	{\GL_{n}(\mathbb{R})/\P_{k_1,\dots,  k_p}(\mathbb{R})} && {\SO_{n}(\mathbb{R})/\S(\O_{n_1}(\mathbb{R}) \times \cdots \times \O_{n_{p+1}}(\mathbb{R}) )}
	\arrow["{\varphi_1}", from=1-1, to=1-3]
	\arrow["{\varphi_2}", from=2-1, to=1-1]
	\arrow["{\varphi_3}"', from=2-3, to=1-3]
\end{tikzcd}\]
with $\varphi_1,\varphi_2,\varphi_3$ defined by 
\begin{gather*}
\varphi_1(\mathbb{W}_1 \subseteq \cdots \subseteq \mathbb{W}_p) = (\mathbb{W}_1,  \mathbb{W}_2/\mathbb{W}_1,  \dots,   \mathbb{W}_p/\mathbb{W}_{p-1},  \mathbb{R}^n/\mathbb{W}_p),   \\
\varphi_2 (\lb X \rb_{\GL}) = (\mathbb{X}_1 \subseteq \cdots \subseteq \mathbb{X}_p),  \qquad
\varphi_3(\lb Q \rb_{\SO}) = (\mathbb{Q}_1,\dots,  \mathbb{Q}_{p+1}),
\end{gather*}
where for each $i=1 ,\dots, p$ and $j = 1,\dots, p+1$,
\begin{enumerate}[\upshape (a)]
\item $\mathbb{X}_i$ is the vector space spanned by the first $k_i$ columns of $X \in \GL_n(\mathbb{R})$;
\item $\mathbb{W}_{i+1}/\mathbb{W}_i$ is the orthogonal complement of $\mathbb{W}_i$ in $\mathbb{W}_{i+1}$, $\mathbb{W}_0 \coloneqq \{0\}$, and $\mathbb{W}_{p+1} \coloneqq \mathbb{R}^n$;
\item $\mathbb{Q}_j$ is the vector space spanned by column vectors of $Q \in \SO_n(\mathbb{R})$ in columns $k_{j-1} + 1,\dots,  k_j$.
\end{enumerate}
\end{lemma}
\begin{proof}
It has been shown in \cite[Proposition~3]{KKL22} that $\varphi_1$ is a diffeomorphism, leaving $\varphi_2$ and $\varphi_3$.  Let $e_1,\dots,  e_n \in \mathbb{R}^n$ be the standard basis vectors and set $\mathbb{E}_j \coloneqq \spn_{\mathbb{R}} \{e_1,\dots, e_j \}$, $j=1,\dots, n$. Consider the map
\[
\rho_2: \GL_n(\mathbb{R}) \to \Flag(k_1,\dots,k_p,\mathbb{R}^n),\quad
X \mapsto X \cdot (\mathbb{E}_{k_1} \subseteq \cdots \subseteq \mathbb{E}_{k_p})
\]
with the action $\cdot$ in \eqref{eq:ac}. This map is surjective as the action of $\GL_n(\mathbb{R})$ is transitive. The stabilizer of $(\mathbb{E}_{k_1} \subseteq \cdots \subseteq \mathbb{E}_{k_p}) \in \Flag(k_1,\dots,  k_p,\mathbb{R}^n) $ is easily seen to be the parabolic subgroup $P_{k_1,\dots, k_p}(\mathbb{R})$ in Section~\ref{sec:mat}. The orbit-stabilizer theorem then yields the required diffeomorphism $\varphi_2$ from $\rho_2$.  The same argument appied to
\[
\rho_3: \SO_n(\mathbb{R}) \to \Flag(k_1,\dots,k_p,\mathbb{R}^n),\quad
Q \mapsto Q \cdot (\mathbb{E}_{k_1} \subseteq \cdots \subseteq \mathbb{E}_{k_p})
\]
yields the required diffeomorphism $\varphi_3$.
\end{proof}
In Section~\ref{sec:coc}, we will add to the list of diffeomorphisms in  Lemma~\ref{lem:diffeo} after establishing various properties of the isospectral model.

\section{Equivariant matrix models of the Grassmannian and flag manifold}\label{sec:or}

We begin by deriving the isospectral model $\Flag_{a_1,\dots,  a_{p+1}}(k_1,\dots,  k_p, n)$, showing that any $\SO_n(\mathbb{R})$-equivariant embedding of $\Flag(k_1,\dots,k_p, \mathbb{R}^n)$ into $\Sym^2 (\mathbb{R}^n)$ must be  an isospectral model. Among these isospectral models are a special class of traceless isospectral models when we choose $a_1,\dots,a_{p+1}$ to satisfy
\[
\sum_{i=0}^p a_{i+1} (k_{i+1} - k_i) = 0.
\]
We have shown in \cite{LK24b} that very $\SO_n(\mathbb{R})$-equivariant model of a flag manifold that is \emph{minimal}, i.e., whose ambient space has the smallest possible dimension, must necessarily be a traceless isospectral model. In other words, no matter what space $\mathbb{V}$ we embed $\Flag(k_1,\dots,k_p, \mathbb{R}^n)$ in, as long as the embedding $\varepsilon$ is equivariant and the dimension of $\mathbb{V}$ is the smallest, then (a) we must have $\mathbb{V} \simeq \Sym^2_\oh (\mathbb{R}^n)$ and (b) the image of $\varepsilon$ must be a traceless isospectral model. The proof of (a) requires a heavy does of representation theory beyond the scope of this article but we will say a few words about it in Corollary~\ref{cor:traceless} to highlight this property for a special case. The claim in (b) follows from Theorem~\ref{thm:classification} and Corollary~\ref{cor:traceless}.

The isospectral model given below in \eqref{eq:iso} appears more complex than the one presented in Section~\ref{sec:intro} but it will be simplified later in Proposition~\ref{prop:simpler model} for generic parameters.
\begin{theorem}[Isospectral model I]\label{thm:classification}
Let $0 \eqqcolon k_0  <   k_1 < \cdots <  k_p  < k_{p+1} \coloneqq n$ be integers and
\[
n_{i+1} \coloneqq k_{i+1} - k_i,\quad i =1,\dots, p.
\]
If $\varepsilon: \Flag(k_1,\dots,  k_p,\mathbb{R}^n) \to \Sym^2 (\mathbb{R}^n)$ is an $\SO_n(\mathbb{R})$-equivariant embedding,  then its image must take the form
\begin{multline}\label{eq:iso}
\Flag_{a_1,\dots,  a_{p+1}}(k_1,\dots,  k_p, n) \coloneqq
\biggl\lbrace
X\in \Sym^2(\mathbb{R}^n): \prod_{j=0}^p (X - a_{j+1} I_n) = 0,  \\[-3ex]
 \tr(X^i) = \sum_{j=0}^p n_{j+1} a_{j+1}^i,\; i = 1,\dots, p
\biggr\rbrace
\end{multline}
for some distinct $a_1,\dots,  a_{p+1} \in \mathbb{R}$.
\end{theorem}
\begin{proof}
Since $\varepsilon$ is $\SO_n(\mathbb{R})$-equivariant,  its image $\varepsilon (  \Flag(k_1,\dots,  k_p,\mathbb{R}^n) )$ is the orbit of a point $B \in \varepsilon (  \Flag(k_1,\dots,  k_p,\mathbb{R}^n) )$ under the action of $\SO_n(\mathbb{R}^n)$ on $\Sym^2(\mathbb{R}^n)$.  Let $b_1 > \cdots > b_{q+1}$ be the distinct eigenvalues of $B$ of multiplicities $m_1,\dots,  m_{q+1}$.  Then we may assume that
\[
B = \begin{bmatrix}
b_1 I_{m_1} & \cdots & 0 \\
\vdots & \ddots & \vdots \\
0 & \cdots & b_p I_{m_{q+1}}
\end{bmatrix}.
\]
It follows from the orbit-stabilizer theorem that
\begin{align*}
\SO_n(\mathbb{R})/\S (\O_{n_1}(\mathbb{R}) \times \cdots \times \O_{n_{p+1}}(\mathbb{R}))  &\simeq \varepsilon (  \Flag(k_1,\dots,  k_p,\mathbb{R}^n) )\\
&\simeq \SO_n(\mathbb{R})/\S (\O_{m_1}(\mathbb{R}) \times \cdots \times \O_{m_{q+1}}(\mathbb{R})),
\end{align*}
from which we deduce that $q = p$ and $\{n_1,\dots,  n_{p+1}\} = \{m_1,  \dots, m_{q+1}\}$.  Let  $\sigma\in \mathfrak{S}_{p+1}$ be such that  $n_1 = m_{\sigma(1)}, \dots, n_{p+1} = m_{\sigma(p+1)}$. Set $a_1 \coloneqq b_{\sigma(1)}, \dots, a_1 \coloneqq b_{\sigma(1)}$. Then $\varepsilon (  \Flag(k_1,\dots,  k_p,\mathbb{R}^n) ) \subseteq  \Flag_{a_1,\dots,  a_{p+1}}(k_1,\dots,  k_p, n)$.  For the reverse inclusion, since $(X - a_1 I) \cdots (X - a_{p+1} I) = 0$, any $X \in  \Flag_{a_1,\dots,  a_{p+1}}(k_1,\dots,  k_p, n)$ has eigenvalues $a_1,\dots,  a_{p+1}$ with multiplicities $n_1,\dots,  n_{p+1}$ determined by the Vandermonde system
\begin{equation}\label{eq:van}
\begin{bmatrix}
1 & 1 & \cdots & 1  & 1 \\
a_1 & a_2 & \cdots & a_p & a_{p+1} \\
a^2_1 & a^2_2 & \cdots & a^2_p & a^2_{p+1} \\
\vdots &  \vdots & \ddots & \vdots & \vdots \\
a^p_1 & a^p_2 & \cdots & a^p_p & a^p_{p+1}
\end{bmatrix}
\begin{bmatrix}
n_1 \\
n_2 \\
n_3 \\
\vdots \\
n_{p+1}
\end{bmatrix} =
\begin{bmatrix}
n  \\
\tr(X) \\
\tr(X^2) \\
\vdots \\
\tr(X^p)
\end{bmatrix}.
\end{equation}
Since we also have that $X \in \Sym^2(\mathbb{R}^n)$, it must take the form $X = Q \diag(a_1 I_{n_1},\dots,  a_{p+1} I_{n_{p+1}})Q^\tp $ for some $Q \in \SO_n(\mathbb{R})$. Hence $\varepsilon (  \Flag(k_1,\dots,  k_p,\mathbb{R}^n) ) \supseteq  \Flag_{a_1,\dots,  a_{p+1}}(k_1,\dots,  k_p, n)$.
\end{proof}

Embedded in the proof of Theorem~\ref{thm:classification} is an alternative parametric characterization of the isospectral model worth stating separately, and which also shows that the object called ``isospectral manifold'' in  \cite{CD91} is exactly a flag manifold.
\begin{corollary}[Isospectral model II]\label{cor:isospectral}
Let $k_0,\dots,  k_{p+1}, n, n_1,\dots, n_{p+1}$ and $a_1,\dots,  a_{p+1}$ be as in Theorem~\ref{thm:classification}. Then
\begin{equation}\label{eq:iso2}
\Flag_{a_1,\dots,  a_{p+1}}(k_1,\dots,  k_p, n)
= \left\lbrace
Q {\setlength{\arraycolsep}{0pt}
\begin{bmatrix}
a_1 I_{n_1} & 0 & \cdots & 0\\[-1ex]
0 &  a_2 I_{n_2}  & &  \vdots \\[-1ex]
\vdots &  & \ddots &  0 \\
0 & \cdots & 0 &  a_{p+1} I_{n_{p+1}}
\end{bmatrix}}
Q^\tp\in \Sym^2(\mathbb{R}^n) : Q\in \SO_n(\mathbb{R}) 
\right\rbrace.
\end{equation}
\end{corollary}

As we alluded to in Section~\ref{sec:intro}, we call \eqref{eq:iso} or \eqref{eq:iso2} the \emph{isospectral model}. There are two special cases worth highlighting separately.
\begin{corollary}[Traceless isospectral model]\label{cor:traceless}
Let $k_0,\dots,  k_{p+1}, n, n_1,\dots, n_{p+1}$ be as in Theorem~\ref{thm:classification}. Let $a_1,\dots,  a_{p+1} \in \mathbb{R}$ be such that
\begin{equation}\label{eq:tf}
\sum_{j=0}^p n_{j+1} a_{j+1} = \sum_{j=0}^p (k_{j+1} - k_j) a_{j+1} = 0.
\end{equation}
Then
\begin{align*}
\Flag_{a_1,\dots,  a_{p+1}}(k_1,\dots,  k_p, n) &=\biggl\lbrace
X\in \Sym^2_\oh(\mathbb{R}^n): \prod_{j=0}^p (X - a_{j+1} I_n) = 0,  \\[-3ex]
&\hspace*{25ex} \tr(X^i) = \sum_{j=0}^p n_{j+1} a_{j+1}^i,\; i = 1,\dots, p
\biggr\rbrace \\
&= \bigl\lbrace
Q \diag(a_1 I_{n_1},  a_2 I_{n_2},\dots,  a_{p+1} I_{n_{p+1}}) Q^\tp\in \Sym^2_\oh(\mathbb{R}^n) : Q\in \SO_n(\mathbb{R}) 
\bigr\rbrace
\end{align*}
has the lowest possible dimension ambient space among all possible $\SO_n(\mathbb{R})$-equivariant models of the flag manifold whenever $n \ge 17$.
\end{corollary}
\begin{proof}
Note that the difference here is that the matrices are traceless, i.e., we have
\[
\Flag_{a_1,\dots,  a_{p+1}}(k_1,\dots,  k_p, n) \subseteq \Sym^2_\oh(\mathbb{R}^n)
\]
and this is obvious from the parametric characterization as any $X \in \Flag_{a_1,\dots,  a_{p+1}}(k_1,\dots,  k_p, n)$ has $\tr(X)$ given by the expression in \eqref{eq:tf}.
The minimal dimensionality and exhaustiveness of these traceless isospectral models when $n \ge 17$ have been established in \cite[Theorem~3.5 and Proposition~3.7]{LK24b}.  Nevertheless, in part as a demonstration of how such a result is plausible, we will prove a special case in Proposition~\ref{prop:k1n2} that avoids group representation theory entirely.
\end{proof}

The $p = 1$ special case of Theorem~\ref{thm:classification} is also worth stating separately. It gives a complete classification of equivariant models of the Grassmannian, namely, they must all be quadratic models.
\begin{corollary}[Quadratic model]\label{cor:quad}
If $\varepsilon: \Gr(k,\mathbb{R}^n) \to \Sym^2 (\mathbb{R}^n)$ is an $\SO_n(\mathbb{R})$-equivariant embedding,  then its image must take the form
\begin{align*}
 \Gr_{a,b}(k,n) &\coloneqq
\bigl\lbrace
X\in \Sym^2(\mathbb{R}^n): (X - aI_n) (X-bI_n) = 0,  \; \tr(X) = ak + b(n-k)
\bigr\rbrace \\
&=
\biggl\lbrace
Q \begin{bmatrix} a I_k & 0 \\ 0 & b I_{n-k} \end{bmatrix} Q^\tp \in \Sym^2(\mathbb{R}^n) :  Q \in \SO_n(\mathbb{R}) \biggr\rbrace
\end{align*}
for some distinct $a,b \in \mathbb{R}$.
\end{corollary}

A notable aspect of the quadratic model is that as $p = 1$, the Vandermonde system \eqref{eq:van} in the defining conditions of \eqref{eq:iso} reduces to a single trace condition. It turns out that generically this reduction always holds, even when $p > 1$. In other words, all we need are the first two rows of \eqref{eq:van},
\begin{align*}
n &= n_1 + \dots + n_{p+1},\\
\tr(X) &= a_1 n_1 + \dots + a_{p+1} n_{p+1}.
\end{align*}
The reason is that $n_1,\dots, n_{p+1}$ are positive integers, which greatly limits the number of possible solutions; in fact for generic values, the two equations above have unique positive integer solutions $n_1,\dots, n_{p+1}$. There is no need to look at the remaining rows of \eqref{eq:van}.

\begin{proposition}[Simpler isospectral model]\label{prop:simpler model}
Let $k_0,\dots,  k_{p+1}, n, n_1,\dots, n_{p+1}$ be as in Theorem~\ref{thm:classification}. For generic $a_1,\dots, a_{p+1} \in \mathbb{R}$,
\[
\Flag_{a_1,\dots,  a_{p+1}}(k_1,\dots,  k_p, n)  =
\biggl\lbrace
X \in \Sym^2(\mathbb{R}^n): \prod_{j=0}^p (X - a_{j+1} I_n) = 0,  \; \tr(X) = \sum_{j=0}^p n_{j+1} a_{j+1}
\biggr\rbrace.
\]
\end{proposition}
\begin{proof}
Let $t \coloneqq  n_1 a_1 + \dots + n_{p+1} a_{p+1}$ and denote the set on the right by $\Fl_{a_1,\dots,  a_{p+1}}(t, n)$. For any distinct $a_1,\dots,  a_{p+1} \in \mathbb{R}$,
\[
\Flag_{a_1,\dots,  a_{p+1}}(k_1,\dots,  k_p, n) \subseteq \Fl_{a_1,\dots,  a_{p+1}}(t, n ).
\]
We will show that the reverse inclusion holds for generic values of $a_1,\dots,a_{p+1}$. Let
\begin{align}
M  &\coloneqq \{(n_1,\dots,  n_{p+1})\in \mathbb{N}^{p+1}:  n_1 + \cdots + n_{p+1} = n\}, \nonumber \\
H & \coloneqq  \{m - m' \in \mathbb{Z}^{p+1}:  m \ne m', \; m, m' \in M \}, \nonumber \\
C  &\coloneqq \{ (a_1,\dots,  a_{p + 1}) \in \mathbb{R}^{p+1}:  a_i \ne a_j \text{ if } i \ne j \}. \label{eq:C}
\end{align}
The hyperplane defined by $h \in H$ is
\[
h^\perp \coloneqq \{ a \in \mathbb{R}^{p+1}:  a^\tp h = a_1 h_1 + \dots + a_{p+1} h_{p+1} = 0\}. 
\]
Since $H$ is a finite set and $C$ is a complement of a union of finitely many hyperplanes, $C \setminus \bigcup_{h\in H } h^\perp$
is also a complement of a union of finitely many hyperplanes. For any $a \in C \setminus \bigcup_{h\in H } h^\perp$ and $m, m' \in M$, $a^\tp m = a^\tp m'$ implies that $m = m'$.  For any $X \in \Fl_{a_1,\dots,  a_{p+1}}(t, n)$, the multiplicities $n_1,\dots,  n_{p+1}$ of $a_1,\dots, a_{p+1}$ are uniquely determined by the single equation $\tr(X) = a_1 n_1 + \dots + a_{p+1}  n_{p+1}$.  Hence $X \in \Flag_{a_1,\dots,  a_{p+1}}(k_1,\dots,  k_p, n) $.
\end{proof}
The genericity assumption on $a_1,\dots, a _{p+1}$ is essential for the simplified description in Proposition~\ref{prop:simpler model} as the following example shows.
\begin{example}
Let $n = 5$,  $p= 2$,  $(a_1,a_2,  a_3) = (0,1,2)$, and consider the isospectral model  $\Flag_{0,1,2}(1,4, 5)$. Then $t = 5$ and we have $\Flag_{0,1,2}(1,4, 5) \subseteq \Fl_{0,1,2}(5,  5)$. To see that this inclusion is strict, take
\[
A = \begin{bmatrix}
0 & 0 & 0 & 0 & 0\\
0 & 1 & 0 & 0 & 0\\
0 & 0 & 1 & 0 & 0\\
0 & 0 & 0 & 1 & 0\\
0 & 0 & 0 & 0 & 2\\
\end{bmatrix},\quad 
B = \begin{bmatrix}
0 & 0 & 0 & 0 & 0\\
0 & 0 & 0 & 0 & 0\\
0 & 0 & 1 & 0 & 0\\
0 & 0 & 0 & 2 & 0\\
0 & 0 & 0 & 0 & 2\\
\end{bmatrix},
\]
and observe that $A \in \Flag_{0,1,2}(1,4, 5) \ne \Flag_{0,1,2}(2,3, 5) \ni B$. We have a disjoint union of nonempty sets
\[
\Flag_{0,1,2}(1,4, 5)  \sqcup \Flag_{0,1,2}(2,3, 5) =  \Fl_{0,1,2}(5,  5),
\]
which implies that
\[
\Flag_{0,1,2}(1,4, 5) \subsetneq \Fl_{0,1,2}(5,  5).
\]
Nevertheless the point of Proposition~\ref{prop:simpler model}  is that there will always be some other choices, in fact uncountably many, of  $a_1,a_2, a_3$ that work. Take $(a_1,a_2,  a_3) = (0,1,\sqrt{2})$. Then $t = 2 \sqrt{2} +1$ and
\begin{align*}
n_1 + n_2 + n_3 &= 5, \\
a_1 n_1 + a_2 n_2 + a_3 n_3 &= 2\sqrt{2} + 1,
\end{align*}
have the unique positive integer solution $(n_1,n_2,n_3) = (2,1,2)$. Hence
\[
\Flag_{0,1,2}(2,3, 5) = \Fl_{0,1,2}(2\sqrt{2} + 1,  5 ).
\]
\end{example}

Special values of $a_1,\dots,  a_{p+1}$ in Theorem~\ref{thm:classification} give us specific models with various desirable features. For instance, we have mentioned the involution model \cite{ZLK20} that parameterizes the Grassmannian with perfectly conditioned matrices, obtained by setting $(a,b)=(1,-1)$ in  Corollary~\ref{cor:quad}.

For $n =2$, $\Flag(1,\mathbb{R}^2) = \Gr(1,\mathbb{R}^2 )$ differs from other cases (because $\SO_2(\mathbb{R})$ is abelian unlike $\SO_n(\mathbb{R})$ for $n \ge 3$) and has to be treated separately. In this case all minimal $\SO_2(\mathbb{R})$-equivariant embeddings of $\Gr(1,\mathbb{R}^2 )$ may be easily characterized.
\begin{proposition}[Minimal equivariant models of  $\Flag(1,\mathbb{R}^2) = \Gr(1,\mathbb{R}^2 )$]\label{prop:k1n2}
Any $\SO_2(\mathbb{R})$-equivariant embedding  of $\Gr(1,\mathbb{R}^2)$ into $ \Sym^2_\oh(\mathbb{R}^2)$ must take the form 
\[
\Gr(1,\mathbb{R}^2) \to \Sym^2_\oh(\mathbb{R}^2), \quad \spn \biggl\{ \begin{bmatrix}
x \\
y
\end{bmatrix} \biggr\} \mapsto \frac{r}{\sqrt{x^2 + y^2}} \begin{bmatrix}
cx - sy & sx + cy \\
sx + cy & -cx + sy
\end{bmatrix}
\]
for some $r > 0$ and some $ \begin{bsmallmatrix}
c & -s \\
s & c
\end{bsmallmatrix} \in \SO_2(\mathbb{R})$. All such embeddings are minimal.  
\end{proposition}
\begin{proof}
Write $\mathbb{S}^1 \coloneqq \{ \begin{bsmallmatrix}
x \\ y
\end{bsmallmatrix}  \in \mathbb{R}^2 : x^2 + y^2 = 1 \}$ for the unit sphere in $\mathbb{R}^2$.  Recall that $\Gr(1,\mathbb{R}^2 ) = \mathbb{RP}^1$, the real projective line.  Let $f:   \mathbb{RP}^1  \to \Sym^2_\oh(\mathbb{R}^2)$ be an $\SO_2(\mathbb{R})$-equivariant embedding.   We consider
\begin{equation}\label{eq:n=1}
j: \mathbb{S}^1 \cong \mathbb{RP}^1 \xrightarrow{f} \Sym^2_\oh(\mathbb{R}^2) \cong \mathbb{R}^2
\end{equation}
where the first $\cong$ is the usual stereographic projection of $\mathbb{S}^1$ onto $\mathbb{R}$ with north pole mapping to the point-at-infinity; and the second $\cong$ is $\Sym^2_\oh(\mathbb{R}^2) \to \mathbb{R}^2 $,  $ \begin{bsmallmatrix} x & y \\ y & -x \end{bsmallmatrix} \mapsto  \begin{bsmallmatrix}
x \\ y
\end{bsmallmatrix}$. It is easy to see that both are  $\SO_2(\mathbb{R})$-equivariant. It remains to characterize all $\SO_2(\mathbb{R})$-equivariant embeddings $j: \mathbb{S}^1  \to \mathbb{R}^2$. 

By equivariance,  we must have 
\[
j(\mathbb{S}^1)  = \left\lbrace \begin{bmatrix}
c & -s \\
s & c
\end{bmatrix} \begin{bmatrix}
x \\
y
\end{bmatrix} \in \mathbb{R}^2:  \begin{bmatrix}
c & -s \\
s & c
\end{bmatrix} \in \SO_2(\mathbb{R})
\right\rbrace,
\] 
where $j(u_0) = \begin{bsmallmatrix}
x_0 \\ y_0
\end{bsmallmatrix} \in \mathbb{R}^2 \setminus \{0\}$ for a fixed $u_0\in \mathbb{S}^1$.  Let
\[ 
r_0 = (x_0^2 + y_0^2)^{1/2} > 0,  \qquad v_0 =  \frac{1}{\sqrt{x_0^2 + y_0^2}} \begin{bmatrix}
x_0 \\
y_0
\end{bmatrix} \in \mathbb{S}^1.
\]
Then $j(\mathbb{S}^1) = r_0 \mathbb{S}^1$.  Let $Q = \begin{bsmallmatrix} 
c & -s \\
s & c
\end{bsmallmatrix} \in \SO_2(\mathbb{R})$ be such that $Q  u_0 = v_0$.  Then $ v_0 =r_0^{-1} j(u_0) =r_0^{-1} j(Q^\tp v_0)$. So the map $ \mathbb{S}^1 \to \mathbb{R}^2$, $v \mapsto r^{-1} j(Q^\tp v)$ is the inclusion $\mathbb{S}^1 \subseteq \mathbb{R}^2$.  Hence we have 
\[
j (v) =  r \begin{bmatrix}
c x - s y \\
s x + c y
\end{bmatrix}
\]
for $v =\begin{bsmallmatrix}
x \\
y
\end{bsmallmatrix} \in \mathbb{S}^1$ and $Q = \begin{bsmallmatrix}
c & -s \\
s & c
\end{bsmallmatrix} \in \SO_2(\mathbb{R})$.
Composing $j$ with the inverses of the two diffeomorphisms in \eqref{eq:n=1},  we obtain
\[
f \biggl( \spn \biggl\{ \begin{bmatrix}
x \\
y
\end{bmatrix} \biggr\} \biggr)  = \frac{r}{\sqrt{x^2 + y^2}} \begin{bmatrix}
cx - sy & sx + cy \\
sx + cy & -cx + sy
\end{bmatrix}
\]
for any $\spn \bigl\{ \begin{bsmallmatrix}
x \\
y
\end{bsmallmatrix} \bigr\} \in \mathbb{RP}^1$.
If the embedding above is not of minimal dimension, then there is an embedding of $\mathbb{S}^1$ into $\mathbb{R}^1$. The image of $\mathbb{S}^1$ in $\mathbb{R}^1$ is connected and compact and thus a closed interval $[a,b]$, a contradiction as $[a,b]$ is contractible and $\mathbb{S}^1$ is not.
\end{proof}

\section{Change-of-coordinates for isospectral models}\label{sec:coc}
 
In this section we provide change-of-coordinates formulas to transform from one model of the flag manifold or Grassmannian to another, focusing on the  isospectral and quadratic models. We begin with the following addendum to Lemma~\ref{lem:diffeo}.
\begin{lemma}[Change-of-coordinates for flag manifold II]\label{lem:change-of-coordinates1}
Let $k_0,\dots,  k_{p+1}, n, n_1,\dots, n_{p+1}$ and $a_1,\dots,  a_{p+1}$ be as in Theorem~\ref{thm:classification}.
Using the parametric characterization in Corollary~\ref{cor:isospectral}, the map
\begin{align*}
\varphi_4:  \Flag_{a_1,\dots,  a_{p+1}}(k_1,\dots,  k_p, n) &\to  \Flag(k_1,\dots,  k_{p};\mathbb{R}^n), \\
Q \diag(a_1 I_{n_1},  a_2 I_{n_2},\dots,  a_{p+1} I_{n_{p+1}}) Q^\tp  &\mapsto  (\mathbb{Q}_1,\dots,  \mathbb{Q}_p),
\end{align*}
where $\mathbb{Q}_j \subseteq \mathbb{R}^n$ is the  subspace spanned by the first $k_j$ column vectors of $Q$, $j =1,\dots,p$, is a diffeomorphism.
\end{lemma}
\begin{proof}
First we need to check that $\varphi_4$ is well-defined. Write
\begin{equation}\label{eq:La}
\Lambda_a \coloneqq \diag(a_1 I_{n_1},  a_2 I_{n_2},\dots,  a_{p+1} I_{n_{p+1}}).
\end{equation}
Since  $Q \Lambda_a Q^\tp = V \Lambda_a V^\tp$ for $Q, V\in \O_n(\mathbb{R})$ if and only if $V = QP$ for some $P \in \S(\O_{n_1}(\mathbb{R}) \times \cdots \times \O_{n_{p+1}}(\mathbb{R}))$ if and only if  $(\mathbb{Q}_1,\dots,\mathbb{Q}_p) = (\mathbb{V}_1,\dots,\mathbb{V}_p)$, which shows that the map $\varphi_4$ is well-defined (and injective too). To see that $\varphi_4$ is a diffeomorphism, observe that $\varphi_4$ factors as
\begin{equation}\label{eq:45}
\begin{tikzcd}
	{\Flag_{a_1,\dots, a_{p+1}}(k_1,\dots, k_p,n)} & {\Flag(k_1,\dots, k_{p};\mathbb{R}^n)} \\
	& {\SO_n(\mathbb{R})/\S(\O_{n_1}(\mathbb{R}) \times \cdots \times \O_{n_{p+1}}(\mathbb{R}))}
	\arrow["{\varphi_4}", from=1-1, to=1-2]
	\arrow["\varphi_5"', from=1-1, to=2-2]
	\arrow["{\varphi_1^{-1} \circ \varphi_3}"', from=2-2, to=1-2]
\end{tikzcd}
\end{equation}
where $\varphi_5$ is defined by $\varphi_5 (Q \Lambda_a Q^\tp) \coloneqq \lb Q \rb$, $\varphi_1$ and $\varphi_3$ are the diffeomorphisms defined earlier in Lemma~\ref{lem:diffeo}.
\end{proof}

In the following, let $C$ be  as in \eqref{eq:C}, the subset of $\mathbb{R}^{p+1}$ comprising elements whose coordinates are all distinct. Note that this set parameterizes all isospectral models of $\Flag(k_1,\dots,  k_{p};\mathbb{R}^n)$. The formula for the transformation between two isospectral models of the same flag manifold or two quadratic models of the same Grassmannian is straightforward.
\begin{proposition}[Change-of-coordinates for isospectral models]\label{prop:changeiso}
Let $k_0,\dots,  k_{p+1}, n, n_1,\dots, n_{p+1}$  be as in Theorem~\ref{thm:classification}. We use the parametric characterization in Corollary~\ref{cor:isospectral} and write $\Lambda_a$ as in \eqref{eq:La} 
for any $(a_1,\dots,  a_{p+1}) \in C$. The map
\[
\varphi : \Flag_{a_1,\dots, a_{p+1}}(k_1,\dots, k_p,n) \to \Flag_{b_1,\dots, b_{p+1}}(k_1,\dots, k_p,n),\quad 
\varphi(Q \Lambda_a Q^\tp) \coloneqq Q \Lambda_b Q^\tp
\]
is an $\SO_n(\mathbb{R})$-equivariant diffeomorphism for any $(a_1,\dots,  a_{p+1})$ and  $(b_1,\dots, b_{p+1}) \in C$. In particular,  for positive integers $k \le n$,
\[
\varphi: \Gr_{a,b}(k,n) \to \Gr_{c,d}(k,n),\quad 
\varphi\biggl(Q \begin{bmatrix}
a I_k & 0 \\
0 & b I_{n-k}
\end{bmatrix} Q^\tp\biggr) \coloneqq Q \begin{bmatrix}
c I_k & 0 \\
0 & d I_{n-k}
\end{bmatrix} Q^\tp,
\]
is an $\SO_n(\mathbb{R})$-equivariant diffeomorphism for any pairs of distinct real numbers $(a,b)$ and $(c,d)$.
\end{proposition}
\begin{proof}
The same argument used in the proof of Corollary~\ref{cor:isospectral} shows that $\varphi$ is well-defined, $\SO_n(\mathbb{R})$-equivariant, bijective, and factors as
\[\begin{tikzcd}
	{\Flag_{a_1,\dots, a_{p+1}}(k_1,\dots, k_p,n)} && { \Flag_{b_1,\dots, b_{p+1}}(k_1,\dots, k_p,n)} \\
	& {\Flag(k_1,\dots,k_p,\mathbb{R}^n)}
	\arrow["\varphi", from=1-1, to=1-3]
	\arrow["{\varphi_a}"', from=1-1, to=2-2]
	\arrow["{\varphi_b^{-1}}"', from=2-2, to=1-3]
\end{tikzcd}\] 
where $\varphi_a$ amd  $\varphi_b$ are the diffeomorphism $\varphi_4$ in Corollary~\ref{cor:isospectral} with respect to $(a_1,\dots,  a_{p+1})$ and  $(b_1,\dots,  b_{p+1}) \in C$.  Hence  $\varphi$ is a diffeomorphism.
\end{proof}

The same proof of Proposition~\ref{prop:changeiso} may be used to establish a stronger result. 
\begin{proposition}[Homotopy for isospectral models]\label{prop:change-of-coordinates1}
With notations as in Proposition~\ref{prop:changeiso}, if
\begin{equation}\label{eq:coc1}
\sign(a_i - a_j) = \sign(b_i - b_j) \quad \text{for all }   i < j,
\end{equation}
then map 
\begin{align*}
\Phi_t: \Flag_{a_1,\dots, a_{p+1}}(k_1,\dots, k_p,n) &\to \Flag_{a_1 + t(b_1 - a_1),\dots, a_{p+1} + t(b_{p+1} - a_{p+1})}(k_1,\dots, k_p,n), \\
Q \Lambda_a Q^\tp &\mapsto \Phi_t(Q \Lambda_a Q^\tp) \coloneqq \Phi(Q \Lambda_a Q^\tp,  t),
\end{align*}
is an $\SO_n(\mathbb{R})$-equivariant diffeomorphism for any $t \in [0,1]$.
\end{proposition}
\begin{proof}
The map $\Phi_t$ in Proposition~\ref{prop:change-of-coordinates1} is a homotopy along the line segment 
\[
\gamma: [0,1] \to \mathbb{R}^{p+1},\quad \gamma(t) =(1-t) (a_1,\dots, a_{p+1}) + t (b_1,\dots, b_{p+1}).
\]
Note that \eqref{eq:coc1} holds if and only if
\[
(1-t) (a_i - a_j) + t(b_i - b_j) \ne 0 \quad \text{for all }  t\in [0,1].
\]
As we saw in the proof of Proposition~\ref{prop:simpler model}, the parameter space $C \subseteq \mathbb{R}^{p+1}$ is disconnected since it is the complement of finitely many hyperplanes.  For arbitrary $(a_1,\dots,  a_{p+1})$ and  $(b_1,\dots,  b_{p+1}) \in C$,  it is possible that $\gamma(t) \not\in C$ for some $t$. The entire curve $\gamma([0,1]) \subseteq C$ if and only if $(a_1,\dots,  a_{p+1})$ and $(b_1,\dots,  b_{p+1})$ are in the same connected component of $C$, which is in turn equivalent to \eqref{eq:coc1}.
\end{proof}

\section{Equivariant matrix models for Stiefel manifolds}\label{sec:V}

The corresponding results for the Stiefel manifold \cite{Stie} are considerably easier. Recall that we write $\V(k, \mathbb{R}^n)$ for the abstract Stiefel manifold of orthonormal $k$-frames in $\mathbb{R}^n$ and $\V(k,n) \coloneqq \{X \in \mathbb{R}^{n \times k} : X^\tp X = I \}$ for its usual model as $n \times k$ orthonormal matrices.

We will call the following family of minimal $\SO_{n}(\mathbb{R})$-equivariant models of $\V(k,\mathbb{R}^n)$,
\begin{equation}\label{eq:Chol}
\V_{\!A} (k,n) \coloneqq\{ Y \in \mathbb{R}^{n \times k}: Y^\tp Y = A \},
\end{equation}
the \emph{Cholesky models}. Clearly it includes the usual model as $\V_{\!I} (k,n) =\V(k,n)$ and is a $G$-manifold via the action
\begin{equation}\label{eq:acVA}
\SO_n(\mathbb{R}) \times \V_{\!A} (k,n) \to \V_{\!A} (k,n), \quad (Q, Y) \mapsto QY.
\end{equation}
The choice of nomenclature and that it is indeed a model of the Stiefel manifold will be self-evident after the next two propositions. We begin with the Stiefel manifold analogue of Theorem~\ref{thm:classification}:
\begin{proposition}[Cholesky model I]\label{prop:classificationStiefel}
Let $k,  n \in \mathbb{N}$ with $k \le n$.  If $\varepsilon: \V(k,\mathbb{R}^n) \to \mathbb{R}^{n\times k}$ is an $\SO_n(\mathbb{R})$-equivariant embedding,  then there is some $A\in \mathsf{S}^2_\pp(\mathbb{R}^k)$ such that 
\[
\varepsilon (\V(k,\mathbb{R}^n)) = \V_{\!A} (k,n).
\]
If $n \ge 17$ and $k < (n-1)/2$,  then $V_{\!A} (k,n)$ has the lowest possible dimension ambient space among all possible $\SO_n(\mathbb{R})$-equivariant models of the Stiefel manifold.
\end{proposition}
\begin{proof}
Since $\varepsilon$ is $\SO_n(\mathbb{R})$-equivariant and $\SO_n(\mathbb{R})$ acts on $\V(k,\mathbb{R}^n)$ transitively via \eqref{eq:acV},  $\varepsilon (\V(k,\mathbb{R}^n)) $ is an $\SO_n(\mathbb{R})$-orbit.  Fix an arbitrary $Y_0 \in \varepsilon (\V(k,\mathbb{R}^n))$. Then
\[
\varepsilon (\V(k,\mathbb{R}^n)) = \left\lbrace
Q Y_0: Q\in \SO_n(\mathbb{R})
\right\rbrace
\subseteq \{Y\in \mathbb{R}^{n\times k}: Y^\tp Y = A \}
\]
where $A \coloneqq Y_0^\tp Y_0 \in \mathsf{S}^2_\pp(\mathbb{R}^k)$.  Conversely,  if $Y^\tp Y = A = Y_0^\tp Y_0$, consider the QR decompositions
\[
Y = Q  \begin{bmatrix}
R  \\
0
\end{bmatrix},\quad 
Y_0 =Q_0 \begin{bmatrix}
R_0 \\
0
\end{bmatrix}
\]
where $Q,  Q_0 \in \O_{n}(\mathbb{R})$ and $R,  R_0 \in \GL_k(\mathbb{R})$ are upper triangular matrices with positive diagonals.  Thus $R^\tp R  = R_0^\tp R_0 = A$ are Cholesky decompositions of $A \in \mathsf{S}^2_\pp(\mathbb{R}^n)$.  By the uniqueness of the Cholesky decomposition of a symmetric positive definite matrix,  we have $R = R_0$ and 
\[
Y = Q  \begin{bmatrix}
R  \\
0
\end{bmatrix}  = Q  \begin{bmatrix}
R _0 \\
0
\end{bmatrix} = (Q Q_0^\tp) \biggl(  Q_0 \begin{bmatrix}
R _0 \\
0
\end{bmatrix}\biggr)
= (Q Q_0^\tp) Y_0 \in \varepsilon (\V(k,\mathbb{R}^n)).
\]
Hence $\varepsilon (\V(k,\mathbb{R}^n)) =  \V_{\!A} (k,n)$. The minimality of $\V_{\!A}(k,n)$ for $n \ge 17$ and $k < (n-1)/2$ follows from \cite[Proposition~3.8]{LK24b}.
\end{proof}
The proof of Proposition~\ref{prop:classificationStiefel} also yields an alternative parametric characterization of the Cholesky model, and may be viewed as the Stiefel manifold analogue of Corollary~\ref{cor:isospectral}.
\begin{corollary}[Cholesky model II]\label{cor:Chol}
Let $k,  n \in \mathbb{N}$ with $k \le n$ and $A\in \mathsf{S}^2_\pp(\mathbb{R}^n)$.  Then
\[
 \V_{\!A} (k,n) = 
\biggl\lbrace
Q \begin{bmatrix}
R \\
0
\end{bmatrix}: Q\in \O_n(\mathbb{R})
\biggr\rbrace
\]
where $R \in \GL_k(\mathbb{R})$ is the unique Cholesky factor of $A\in \mathsf{S}^2_\pp(\mathbb{R}^n)$.  
\end{corollary}

Lastly we present the change-of-coordinate formula for Cholesky models, which will also provide a pretext for discussing some interesting features of its parameter space. We borrow the shorthand in \cite[Equation~21]{BH06} and write, for any $A,  B\in \mathsf{S}^2_\pp(\mathbb{R}^k)$ and $t \in [0,1]$,
\[
A \#_t B = A^{1/2} (A^{-1/2} B A^{-1/2} )^t A^{1/2}.
\]
The special case when $t =1/2$ gives  $A \#_{1/2} B = A^{1/2} (A^{-1/2} B A^{-1/2} )^{1/2} A^{1/2} \eqqcolon  A \# B$, the matrix geometric mean \cite[Equation~1]{BH06}. The Stiefel manifold analogue of  Propositions~\ref{prop:changeiso} and \ref{prop:change-of-coordinates1} is as follows.
\begin{proposition}[Homotopy and change-of-coordinates for Cholesky models]\label{prop:change-of-coordinates2}
Let $k,  n \in \mathbb{N}$ with $k \le n$. For any $A,  B\in \mathsf{S}^2_\pp(\mathbb{R}^k)$ and $t \in [0,1]$,
\begin{align*}
\psi_t : \V_{\!A}(k,n) &\to \V_{\!A \#_t B}(k,n),\\
Y &\mapsto Y A^{-1/2} (A^{-1/2} B A^{-1/2} )^{t/2} A^{1/2},
\end{align*}
is an $\SO_n(\mathbb{R})$-equivariant diffeomorphism. In particular the map $\psi_1$ gives a change-of-coordinates formula from $\V_{\!A}(k,n)$ to $\V_{\!B}(k,n)$.
\end{proposition}
\begin{proof}
The map $\psi_t$ is well-defined as 
\[
( Y A^{-1/2} (A^{-1/2} B A^{-1/2} )^{t/2} A^{1/2})^\tp Y A^{-1/2} (A^{-1/2} B A^{-1/2} )^{t/2} A^{1/2} = A \#_t B
\]
whenever $Y\in \V_{\!A}(k,n)$. Since both $A$ and $B$ are positive definite,  $A^{-1/2} (A^{-1/2} B A^{-1/2} )^{t/2} A^{1/2}$ is invertible, so $\psi_t$ is a diffeomorphism. It is evidently $\SO_n(\mathbb{R})$-equivariant under \eqref{eq:acVA}. 
\end{proof}
Proposition~\ref{prop:change-of-coordinates2} is considerably simpler than Proposition~\ref{prop:change-of-coordinates1} as the parameter space $\mathsf{S}^2_\pp(\mathbb{R}^n)$ is path connected and any two points $A, B \in \mathsf{S}^2_\pp(\mathbb{R}^n)$ can be connected by a curve
\begin{equation}\label{eq:CarG}
\gamma:[0,1] \to \mathsf{S}^2_\pp(\mathbb{R}^n),\quad \gamma(t) = A \#_t B.
\end{equation}
Readers may recognize $\gamma$ as the geodesic curve from $A$ to $B$ in the Cartan manifold \cite[pp.~364--372]{Cartan1}, i.e., $\mathsf{S}^2_\pp(\mathbb{R}^n)$ equipped with the Riemannian metric $\mathsf{g}_A(X,Y) = \tr(A^{-1}XA^{-1}Y)$, or, as an abstract manifold, the set of ellipsoids in $\mathbb{R}^n$ centered at the origin. See \cite{Mos} and \cite[Section~3]{MosBook} for a modern exposition, and \cite[Section~8]{ZLK24} for further bibliographical references about the Cartan manifold.

\section{Equivariant Riemannian metrics}\label{sec:metric}

Here we will discuss $\SO_n(\mathbb{R})$-invariant Riemannian metrics for the flag, Grassmann, and Stiefel manifolds that go alongside their models in this article. As we alluded to at the end of Section~\ref{sec:intro}, each of these manifolds comes equipped with a god-given Riemannian metric. This is a result of their structure as $G/H$ with $G$ a compact simple Lie group and $H$ a closed subgroup.  Such a $G$ admits a bi-invariant metric, unique up to a constant factor,  which in turn induces an invariant metric on $G/H$. If their corresponding Lie algebras $\mathfrak{g}$ and $\mathfrak{h}$ are related by $\mathfrak{g} = \mathfrak{h} \oplus \mathfrak{m}$  where $\mathfrak{m}$ is an $\Ad_H$-invariant subspace of $\mathfrak{g}$,  then there is a one-to-one correspondence
\begin{equation}\label{eq:metric correspondence}
\lbrace
\text{$G$-invariant metrics on $G/H$}
\rbrace \longleftrightarrow 
\lbrace
\text{$H$-invariant metrics on $\mathfrak{m}$}
\rbrace,
\end{equation}
a standard result in differential geometry \cite[Proposition~3.16]{CE08}.  In particular,  we have 
\begin{equation}\label{eq:tan}
\mathfrak{m}  \simeq  \mathbb{T}_{\lb e \rb} G/H,
\end{equation}
and $e\in G$ the identity element.

Slightly less standard (we are unable to find a reference for the simple form below) is the following simple criterion for an equivariant embedding to be isometric. 
\begin{lemma}\label{lem:metrics}
Let $\mathbb{V}$ be a $G$-module and $\varphi: G/H \to \mathbb{V}$ be a $G$-equivariant embedding. Let $\mathsf{g}$ and $\mathsf{g}'$ be Riemannian metrics on $G/H$ and $\mathcal{M} \coloneqq \varphi (G/H) \subseteq \mathbb{V}$ respectively. Consider the linear map
\begin{equation}\label{eq:f}
f: \mathfrak{m} \simeq \mathbb{T}_{\lb e \rb} G/H \xrightarrow{d_{\lb e \rb} \varphi } \mathbb{T}_{ \varphi(\lb e \rb) } \mathcal{M}.
\end{equation}
Then $\varphi$ is isometric if and only if $f$ is an isometry. 
\end{lemma}
\begin{proof}
Since $\varphi$ is $G$-equivariant and $G$ acts on both $G/H$ and $\mathcal{M}$ transitively,  $\varphi$ is isometric if and only if $\varphi$ is isometric at $\lb e \rb$,  i.e.,  the differential map $d_{\lb e \rb} \varphi$ of $\varphi$ at $\lb e \rb$ is an isometry. By \eqref{eq:metric correspondence}, this is equivalent to $f$ being an isometry.
\end{proof}

For the cases of interest to us,
\begin{align*}
\V(k, \mathbb{R}^n) &\cong  \SO_n(\mathbb{R})/\SO_{n-k}(\mathbb{R}),\\
\Gr(k,\mathbb{R}^n) &\cong \SO_{n}(\mathbb{R})/\S(\O_{k}(\mathbb{R}) \times \O_{n-k}(\mathbb{R}) ),\\
\Flag(k_1,\dots,  k_p,\mathbb{R}^n) &\cong \SO_{n}(\mathbb{R})/\S(\O_{n_1}(\mathbb{R}) \times \cdots \times \O_{n_{p+1}}(\mathbb{R}) ).
\end{align*}
We will show that the restriction of the Euclidean inner product on $\mathbb{R}^{n \times k}$ and $\Sym^2(\mathbb{R}^n)$ onto the Cholesky, quadratic, and isospectral models give the Riemannian metric induced by the bi-invariant metric on  $G = \SO_n(\mathbb{R})$ up to a choice of weights.

\subsection{Quadratic model of the Grassmannian}\label{sec:metG}

The bi-invariant metric on $\SO_n(\mathbb{R})$ induces an $\SO_n(\mathbb{R})$-invariant metric $\mathsf{g}$ on $\SO_{n}(\mathbb{R})/\S(\O_{k}(\mathbb{R}) \times  \O_{n - k}(\mathbb{R}) )$, and thus on $\Gr_{a,b}(k,n)$ via $\varphi_5$ in \eqref{eq:45}, whose inverse is
\begin{equation}\label{eq:phi5}
\varphi_5^{-1}: \SO_{n}(\mathbb{R})/\S(\O_{k}(\mathbb{R}) \times  \O_{n - k}(\mathbb{R}) ) \to \Gr_{a,b}(k,n),\quad \varphi_5^{-1} (  \lb Q \rb) = Q \begin{bmatrix}
a I_k & 0 \\
0 & b I_{n-k}
\end{bmatrix} Q^\tp.
\end{equation}
We recall from  \cite[Proposition~5]{KKL22} that
\begin{align*}
\Alt^2(\mathbb{R}^n) &=  \Alt^2(\mathbb{R}^k) \oplus \Alt^2(\mathbb{R}^{n-k})  \oplus \mathfrak{m}, \\
\mathfrak{m} &= \biggl\lbrace
B = \begin{bmatrix}
0 & B_0 \\
-B_0^\tp & 0  
\end{bmatrix}\in \Alt^2(\mathbb{R}^n):  B_0 \in \mathbb{R}^{k \times (n-k)}
\biggr\rbrace,
\end{align*}
and that $\mathfrak{m}$ is invariant under conjugation by $\S(\O_{k}(\mathbb{R}) \times  \O_{n - k}(\mathbb{R}) )$.  The $\S(\O_{k}(\mathbb{R}) \times  \O_{n - k}(\mathbb{R}))$-invariant inner product on $\mathfrak{m}$ is given by
\begin{equation}\label{eq:ip1}
\langle B,  C \rangle \coloneqq (a-b)^2 \tr(B^\tp C) = 2 (a-b)^2 \tr(B_0^\tp C_0)
\end{equation}
for any $B,C\in \mathfrak{m}$. This corresponds, via \eqref{eq:metric correspondence}, to an $\SO_n(\mathbb{R})$-invariant metric on $\SO_{n}(\mathbb{R})/\S(\O_{k}(\mathbb{R}) \times  \O_{n - k}(\mathbb{R}) )$ and differs from the god-given metric $\mathsf{g}$ on $\Gr(k,\mathbb{R}^n)$ by a weight constant.

The standard  Euclidean metric $\tr(XY)$ on $\mathsf{S}^2(\mathbb{R}^n)$ restricts to a metric $\mathsf{g}'$ on $\Gr_{a,b}(k,n)$. We will see that $\mathsf{g}$ and $\mathsf{g}'$ are one and the same, up to a weight constant.
\begin{proposition}\label{cor:Rie}
Let $\SO_{n}(\mathbb{R})/\S(\O_{k}(\mathbb{R}) \times  \O_{n - k}(\mathbb{R}) )$ and $\Gr_{a,b}(k,n)$ be equipped with Riemannian metrics $\mathsf{g}$ and $\mathsf{g}'$ respectively.  Then $\varphi_5^{-1}$ is an isometric $\SO_n(\mathbb{R})$-equivariant diffeomorphism.
\end{proposition}
\begin{proof}
The $\SO_n(\mathbb{R})$-equivariance of $\varphi_5^{-1}$ is evident from its definition.  Let $f$ be the linear map in \eqref{eq:f} for $\varphi_5^{-1}$.  Then by \eqref{eq:phi5},
\[
f(B) = B \begin{bmatrix}
a I_k & 0 \\
0 & b I_{n-k}
\end{bmatrix} + \begin{bmatrix}
a I_k & 0 \\
0 & b I_{n-k}
\end{bmatrix} B^\tp = (b-a) B
\]
for any $B \in \mathfrak{m}$. By \eqref{eq:ip1},
\[
\langle B,  C \rangle = 2 (a-b)^2 \tr(B_0^\tp C_0) = \mathsf{g}'(f(B),  f(C)).
\]
So $f$ is an isometry and hence so is $\varphi_5^{-1}$ by Lemma~\ref{lem:metrics}.
\end{proof}

\subsection{Cholesky model of the Stiefel manifold}

The bi-invariant metric on $\SO_n(\mathbb{R})$ induces an $\SO_n(\mathbb{R})$-invariant metric $\mathsf{g}$ on $\SO_{n}(\mathbb{R})/\SO_{n - k}(\mathbb{R}) )$, and thus on $\V_{\!A}(k,n)$ via
\begin{equation}\label{eq:psiA}
\psi_A: \SO_n(\mathbb{R})/\SO_{n-k}(\mathbb{R}) \to \V_{\!A}(k,n), \quad
\psi_A ( \lb Q \rb) = Q \begin{bmatrix}
R \\
0
\end{bmatrix},
\end{equation}
which is a diffeomorphism by Proposition~\ref{prop:classificationStiefel}. Here $R \in \GL_k(\mathbb{R})$ is the Cholesky factor of $A = R^\tp R  \in \mathsf{S}^2_\pp(\mathbb{R}^n)$.
We have
\begin{align*}
\Alt^2(\mathbb{R}^n) &= \Alt^2(\mathbb{R}^{n-k})  \oplus \mathfrak{m}, \\
\mathfrak{m} &= \biggl\lbrace
\begin{bmatrix}
B_{1} & -B_{2}^\tp \\
B_{2} & 0
\end{bmatrix} \in \mathbb{R}^{n\times n}: B_1 \in \Alt^2(\mathbb{R}^k), \; B_2 \in \mathbb{R}^{(n-k) \times k}
\biggr\rbrace,
\end{align*}
where we identify $B \in  \Alt^2(\mathbb{R}^{n-k}) $ with $\diag(0, B) \in  \Alt^2(\mathbb{R}^n)$.
It is clear that $\mathfrak{m}$ is invariant under conjugation by $\SO_{n-k}(\mathbb{R})$, where we identify $Q\in \SO_{n-k}(\mathbb{R})$ with $\diag(I_k,  Q) \in \SO_n (\mathbb{R})$. For $A = R^\tp R \in \mathsf{S}^2_\pp(\mathbb{R}^n)$, the $\SO_{n-k}(\mathbb{R})$-invariant $A$-inner product on $\mathfrak{m}$ is given by
\begin{equation}\label{eq:weighted metric stiefel}
\langle B,  C \rangle \coloneqq  \tr(R^\tp (B_1^\tp C_1 + B_2^\tp C_2) R)
\end{equation}
for any $B,  C\in \mathfrak{m}$. This correspond, via \eqref{eq:metric correspondence}, to an $\SO_n(\mathbb{R})$-invariant metric $\mathsf{g}_A$ on $\SO_n(\mathbb{R})/\SO_{n-k}(\mathbb{R})$ and differs from the god-given metric $\mathsf{g}$ on $\V(k,\mathbb{R}^n)$ by a weight matrix $A$.  In particular, for $A = I$, we have $\mathsf{g}_{I} = \mathsf{g}$.

The standard Euclidean metric $\tr(X^\tp Y)$ on $\mathbb{R}^{n \times k}$ restricts to a metric $\mathsf{g}'$ on $\V_{\!A}(k,n)$.  We will see that $\mathsf{g}$ and $\mathsf{g}'$ are one and the same, up to a weight matrix.
\begin{proposition}[Riemannian metric on the Stiefel manifold]\label{prop:Rie stiefel}
Let $\SO_n(\mathbb{R})/\SO_{n-k}(\mathbb{R})$ and $\V_{\!A}(k,n)$ be equipped with Riemannian metrics $\mathsf{g}_A$ and $\mathsf{g}'$ respectively.  
Then $\psi_A$ is an isometric $\SO_n(\mathbb{R})$-equivariant diffeomorphism.
\end{proposition}
\begin{proof}
The $\SO_n(\mathbb{R})$-equivariance of $\psi_A$ is evident from its definition. Let $f$ be the linear map in \eqref{eq:f} for $\psi_A$.  Then by \eqref{eq:psiA},
\[
f(B) = B \begin{bmatrix}
R \\
0
\end{bmatrix} = \begin{bmatrix}
B_1  \\
-B_2^\tp 
\end{bmatrix}R,
\]
for any $B = \begin{bsmallmatrix}
B_{1} & -B_{2}^\tp \\
B_{2} & 0
\end{bsmallmatrix}\in \mathfrak{m}$. By \eqref{eq:weighted metric stiefel},
\[
\langle B,  C \rangle = \tr(R^\tp (B_1^\tp C_1 + B_2^\tp C_2) R ) = \mathsf{g}'( f(B),  f(C) ).
\]
So $f$ is an isometry and hence so is $\psi_A$ by Lemma~\ref{lem:metrics}.
\end{proof}

\subsection{Isospectral model of the flag manifold}

The argument here is similar to that of Section~\ref{sec:metG}, but involves heavier notations, and as such we think it is instructive to include the special case in Section~\ref{sec:metG} for clarity.
The bi-invariant metric on $\SO_n(\mathbb{R})$ induces an $\SO_n(\mathbb{R})$-invariant metric $\mathsf{g}$ on  $\SO_{n}(\mathbb{R})/\S(\O_{n_1}(\mathbb{R}) \times \cdots \times \O_{n_{p+1}}(\mathbb{R}) )$, and thus on $\Flag_{a_1,\dots, a_{p+1}}(k_1,\dots, k_p,n)$ via $\varphi_5$ in \eqref{eq:45}, whose inverse is
\[
\varphi_5^{-1} : \SO_n(\mathbb{R})/\S(\O_{n_1}(\mathbb{R}) \times \cdots \times \O_{n_{p+1}}(\mathbb{R})) \to \Flag_{a_1,\dots, a_{p+1}}(k_1,\dots, k_p,n), \quad
 \varphi_5^{-1} (  \lb Q \rb) \coloneqq Q \Lambda_a Q^\tp,
\]
where $\Lambda_a$ as in \eqref{eq:La}. 
We recall from \cite[Proposition~5]{KKL22} that
\begin{align*}
\Alt^2(\mathbb{R}^n) &= \Alt^2(\mathbb{R}^{n_1})\oplus \dots \oplus \Alt^2(\mathbb{R}^{n_{p+1}}) \oplus \mathfrak{m},\\
\mathfrak{m} &= \lbrace
(B_{ij}) \in \mathbb{R}^{n\times n}: B_{ij} = -B_{ji}^\tp \in \mathbb{R}^{n_i  \times n_j},\; B_{ii} = 0,  \; 1 \le i < j \le p+1 
\rbrace,
\end{align*}
and that $\mathfrak{m}$ is invariant under conjugation by $\S(\O_{n_1}(\mathbb{R}) \times \cdots \times \O_{n_{p+1}}(\mathbb{R}))$.  The  $\S(\O_{n_1}(\mathbb{R}) \times \cdots \times \O_{n_{p+1}}(\mathbb{R}))$-invariant inner product on $\mathfrak{m}$ is given by
\begin{equation}\label{eq:weighted metric}
\langle B,  C \rangle \coloneqq  2 \sum_{1 \le i < j \le p+1}(a_i - a_j)^2\tr(B_{ij}^\tp C_{ij}),  
\end{equation}
where $B,C\in \mathfrak{m} \subseteq \Alt^2(\mathbb{R}^n)$ are partitioned as $B = (B_{ij})$, $C = (C_{ij})$ with $B_{ij}$,  $C_{ij} \in \mathbb{R}^{n_i \times n_j}$ for $i, j \in \{1,\dots, p+1\}$.   This corresponds via \eqref{eq:metric correspondence} to  an $\SO_n(\mathbb{R})$-invariant metric $\mathsf{g}_a$ on $\SO_{n}(\mathbb{R})/\S(\O_{n_1}(\mathbb{R}) \times \cdots \times \O_{n_{p+1}}(\mathbb{R}) )$ and differs from the god-given metric $\mathsf{g}$ on $\Flag(k_1,\dots,  k_p,\mathbb{R}^n)$ by a weight vector $a \coloneqq (a_1,\dots,  a_{p+1})$.

The standard  Euclidean metric $\tr(XY)$ on $\mathsf{S}^2(\mathbb{R}^n)$ when restricted to $\Flag_{a_1,\dots, a_{p+1}}(k_1,\dots, k_p,n)$ gives a metric $\mathsf{g}'$. We will see that $\mathsf{g}$ and $\mathsf{g}'$ are one and the same, up to a weight vector.
\begin{proposition}[Riemannian metric on the flag manifold]\label{prop:Rie}
Let $\SO_n(\mathbb{R})/\S(\O_{n_1}(\mathbb{R}) \times \dots \times \O_{n_{p+1}}(\mathbb{R}))$ and $\Flag_{a_1,\dots, a_{p+1}}(k_1,\dots, k_p,n)$ be equipped be Riemannian metrics $\mathsf{g}_a$ and $\mathsf{g}'$ respectively.
Then $\varphi_5^{-1}$ is an isometric $\SO_n(\mathbb{R})$-equivariant diffeomorphism.
\end{proposition}
\begin{proof}
The $\SO_n(\mathbb{R})$-equivariance of $\varphi_5^{-1}$ is evident from its definition.  Let $f$ be the linear map defined in Lemma~\ref{lem:metrics} for $\varphi_5^{-1}$.  Then by \eqref{eq:45},
\[
f(B) = B\Lambda_a + \Lambda_a B^\tp = [(a_j - a_i) B_{ij} ]_{i,j=1}^{p+1}
\]
for any $B \in \mathfrak{m}$. By \eqref{eq:weighted metric},
\[
\langle B,  C\rangle  = 2 \sum_{1\le i < j \le p+1} (a_i - a_j)^2 \tr(B_{ij}^\tp C_{ij}) =\mathsf{g}'(f(B), f(C))
\]
So $f$ is an isometry and hence so is $\varphi_5^{-1}$   by Lemma~\ref{lem:metrics}.
\end{proof}

\section{Conclusion}

We used to be able to count on one hand the number of different models for each of these manifolds. With these families of models, we now have uncountably many choices, and having such flexibility can provide a real benefit as different models are useful in different ways.

Take the family of quadratic models $\Gr_{a,b}(k,n)$ for example. The traceless model with $(a,b) = (n-k,k)$ in Corollary~\ref{cor:traceless} has the lowest dimension but the involution model with $(a,b) = (1, -1)$ in \cite{ZLK20} has the best condition number. It may appear that the projection model with $(a,b)=(1,0)$ makes the worst choice from a computational perspective since it is, up to a constant, the only model in the family with singular matrices. However we found in \cite{ZLK24} that it is the most suitable model for discussing computational complexity issues, as many well-known NP-hard problems have natural formulations as optimization problems in the projection model.

We hope that these families of models for various manifolds described and classified in this article would provide useful computational platforms for practical applications involving these manifolds. 

\bibliographystyle{abbrv}

\end{document}